\documentclass[12pt]{article}
\usepackage{fancyhdr,graphicx}
\usepackage{amsfonts}
\usepackage{amssymb}
\usepackage{color}
\usepackage{amsthm}
\usepackage{newlfont}
\usepackage{amsmath}
\usepackage{bbm}

\newtheorem{thm}{Theorem}[section]
\newtheorem{lemma}[thm]{Lemma}
\newtheorem{cor}[thm]{Corollary}
\newtheorem{pro}[thm]{Proposition}
\newtheorem{conj}[thm]{Conjecture}

\title{A relation between Clar covering polynomial and cube polynomial\footnote{This
 work is  supported by  NSFC (10011083).}}

\author{Heping Zhang$^a$, Wai Chee Shiu$^b$, Pak Kiu Sun$^b$}

\date{\small $^a$School of Mathematics and Statistics, Lanzhou
University, Lanzhou,  Gansu 730000, China, E-mail: zhanghp@lzu.edu.cn\\
\vskip 2mm
$^b$Department of Mathematics, Hong Kong Baptist University, 224 Waterloo Road, Kowloon Tong, Hong Kong, China, E-mails: wcshiu@hkbu.edu.hk, lionel@hkbu.edu.hk}

\begin{document}

\maketitle
\begin{abstract}
The Clar covering polynomial (also called Zhang-Zhang polynomial in some chemical literature) of a hexagonal system is a counting polynomial for some types of resonant structures called Clar covers, which can be used to determine Kekul\'e count, the first Herndon number and Clar number, and so on. In this paper we find that the Clar covering polynomial of a hexagonal system $H$ coincides with the cube polynomial of its resonance graph $R(H)$ by establishing a one-to-one correspondence between the Clar covers of $H$ and the hypercubes in $R(H)$.  Accordingly, some applications are presented.

\end{abstract}

\section{Introduction}

A hexagonal or benzenoid system is a 2-connected finite plane graph such that
every interior face is a regular hexagon of side length one. It can also
 be formed by a cycle with its interior in the infinite hexagonal lattice on the plane (graphene).
A hexagonal system with a Kekul\'e structure (or perfect matching in graph theory) is called  Kekul\'ean and is viewed as the carbon-skeleton
of a benzenoid hydrocarbon. Based on the Kekul\'e structures, there are several classical models of benzenoid molecules relating with resonance energy (extra stability) such as Clar's aromatic sextet
theory \cite{clar1972} and Randi\'c conjugated circuit model \cite{r}.

In 1996, Zhang and  Zhang \cite{zz1} introduced the Clar covering polynomial
(Zhang-Zhang polynomial) of a hexagonal system, which
unifies some topological indices such as the Clar number, the
Kekul\'e count and the first Herndon number. Moreover, it is closely related to \cite{zz2} sextet polynomial, which introduced by Hosoya and Yamaguchi\cite{hy}.

A {\it generalized hexagonal systems} $H$ is a subgraph of a hexagonal system and a {\em Clar cover} is a spanning subgraph of $H$ such that every component of it is either a hexagon or an edge. The set of Clar covers of $H$ is called a {\em sextet pattern} or {\em resonant set} and a {\em Clar formula} is a maximum sextet pattern of $H$. Moreover, the {\em Clar number} $Cl(H)$ of $H$ is the size of a {\em Clar formula} of $H$. The \emph{Clar covering polynomial} of $G$ is defined as follows:

\begin{equation}\zeta(H,x)=\zeta(H)=\sum_{k=0}^{ Cl(H)}  z(H,k)x^{k},\end{equation}
where $z(H,k)$ is the number of Clar covers with $k$ hexagons of $H$.

The dependence of the topological resonance energy on  $\zeta(H,x)$ for some values of $x$ was examined in a series of papers \cite{GGRV,ggf,GGSF}. A recent survey \cite{ZGZ} and some articles \cite{CDG,CLW, CW,CW2,GB,gfb,ZJY} established basic properties of Clar covering polynomial and methods to compute it. In particular, Chou et al. \cite{CLW, CW} recently carried out an automatic computation program for the Clar covering polynomials of benzenoids and obtained many fruitful results.

The resonance graph $R(H)$ (also called Z-transformation graph) of a hexagonal system $H$ was introduced independently by Gr\"undler \cite{gw}, Zhang et al. \cite{zgc,zgc2} and Randi\'c \cite{r,r2}. It originates from Herndon's resonance theory \cite{h} to reflect an interaction between Kekul\'e structure of benzenoid hydrocarbons.  The vertex set of $R(H)$ is the set of perfect matchings of $H$. Two vertices being adjacent if  their symmetric difference forms a hexagon of $H$. It is evident that $z(H,0)$ equals the number of the vertices of $R(H)$ and $z(H,1)$, the first Herndon number, equals the number of edges of $R(H)$. The resonance graph was later extended \cite{zz} to bipartite plane graphs and \cite{z3} gives a survey on it recently.

The $n$-dimensional hypercube $Q_n$ (or simply $n$-cube) is the graph whose vertices are all binary strings of length $n$  and two vertices are adjacent if their strings differ exactly in one position. Bre\v{s}ar, Klav\v{z}ar and \v{S}krekovski \cite{BKS} introduced a counting polynomial of hypercubes of a graph $G$ as follows:
\begin{equation}C(G,x)=\sum_{i\geq 0}\alpha_i(G)x^i,
\end{equation}
where $\alpha_i(G)$ denotes the number of induced subgraphs of $G$ that is isomorphic to the $i$-cube $Q_i$.

An important new role of hypercubes of resonance graph was introduced in \cite{M,SKG,TZ}. Klav\v{z}ar et al.  \cite{KZG} proved that the Clar number $Cl(H)$ of $H$ is equal to the  largest $i$ such that $Q_i$ is a subgraph of $R(H)$.  Zhang and Zhang  \cite{zz1} showed that $z(H,Cl(H))$ equals the number of Clar formulas of $H$; and Salem et al. \cite{SKVZ} established a bijection between the Clar formulas of $H$ and the number of the largest hypercubes in $R(H)$.  These two results together imply that $z(H,Cl(H))$ equals to the number of largest hypercubes in $R(H)$.

To generalize the previous result,   we show that the Clar covering polynomial of a hexagonal system $H$ coincides with the cube polynomial of its resonance graph $R(H)$, that is \begin{equation}\label{eq.}\zeta(H,x)=C(R(H),x)).
\end{equation}
In particular, the coefficients of these two polynomials must be coincided and thus, $z(H,i)=\alpha_i(R(H))$ for each $i\geq 0$. The proof is accomplished by establishing a one-to-one correspondence between the Clar covers of $H$ and the hypercubes in $R(H)$.

For example, the Clar covering polynomial of a fibonacene with $n$ hexagons is equal to the cube polynomial of Fibonacci cube of order $n$. This correspondence also derives an isomorphism between partial orderings on the  Clar covers of $H$ and the hypercubes of $R(H)$.
By using such connection, some further applications are presented. From cube polynomial  the derivatives and properties of real roots of the Clar covering polynomial are produced; Motivated from Clar covering polynomial we derive a novel expression of cube polynomial of a median graph in power of $(x+1)$.  A more general conjecture about the unimodality of the sequence of coefficients of cube polynomial of a median graph is put forward.

 \section{An equality of two polynomials}

The objective of this section is to prove Equation (\ref{eq.}) and the following definitions can simplify our proof.

We first recall two important definitions from set theory and graph theory. The {\em symmetry difference} $A\oplus B$ of two sets $A$ and $B$ is equal to $(A\setminus B)\cup (B\setminus A)$. Let $G$ be a graph with a perfect matching $M$, a cycle of $G$ is $M$-{\em alternating} if its edges belong alternately to $M$ and not to $M$. Also, the edge set of $G$ is denoted by $E(G)$.

  Given any Kekul\'ean hexagonal system $H$. Let $\mathbb{Z}(H,n)$ be the set of Clar covers of $H$ with exactly $n$ hexagons. On the other hand, consider a graph $G$. The set of induced subgraphs of $G$ that isomorphic to $n$-cube $Q_n$ is denoted by $\mathbb{Q}_n(G)$. Combining these definitions from those in previous section, we have $z(H,n)=|\mathbb{Z}(H,n)|$ and $\alpha_n(G)=|\mathbb{Q}_n(G)|$.

To prove Equation (\ref{eq.}), it is sufficient to establish a bijection between $\mathbb{Z}(H,n)$ and $\mathbb{Q}_n(R(H))$ for each integer $n\geq 0$. To achieve our goal, let
\begin{equation} \label{map_eq}f: \mathbb{Z}(H,n)\rightarrow\mathbb{Q}_n(R(H))
\end {equation}
be a mapping and it is defined as follows: For each Clar cover $C\in \mathbb{Z}(H,n)$, consider those perfect matchings $M_1, M_2, \dots M_i$ in $H$ such that each hexagon in $C$ is $M_j$-alternating and each isolated edge in $C$ is in $M_j$ for all $j \in [1,i]$. Assign $f(C)$ as the induced subgraph of $R(H)$ with vertices $M_1, M_2, \dots M_i$.

\begin{figure}
\centering
\includegraphics[scale=0.6]{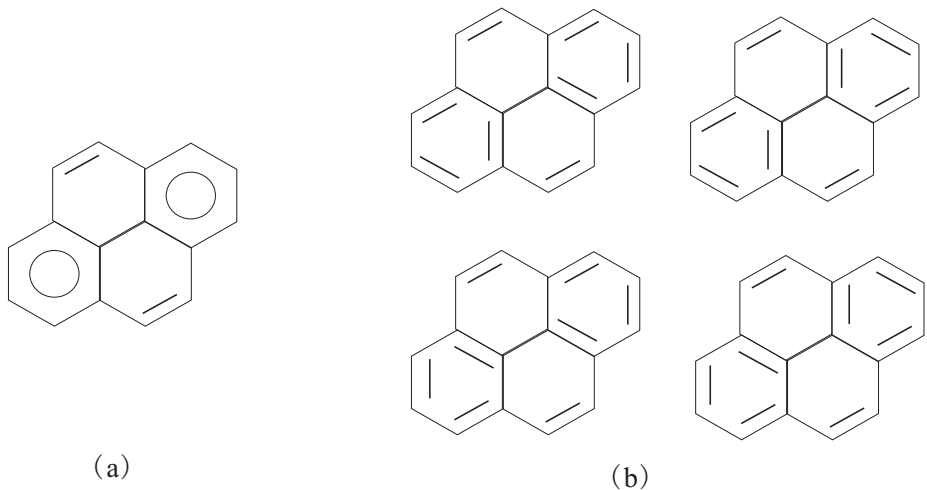}
\caption{(a) A Clar cover $C$ of pyrene, (b) The four perfect matchings in the image $f(C)$ which induces a square in the resonance graph.}\label{fig_map}
\end{figure}

 It is not difficult to show that $M_j \oplus E(C)$ only consists of perfect matchings of the hexagons in $C$.  Since each hexagon in $C$ has two perfect matchings, thus, $f(C)$ has exactly $2^n$ vertices. An example illustrating the definition of $f$ is given in Figure \ref{fig_map}; the hexagons with a circle inside represent the hexagons of $C$ and a double bond represent an isolated edge of $C$.

To facilitate the proof follows, we orientate the plane graph $H$ such that some edges are vertical. For a perfect matching $M$ of $H$, an $M$-alternating hexagon $s$ is called a {\em proper}  ({\em resp. improper}) {\em sextet} if the vertical double edge lies on the right (resp. left) (see Figure 2).

The following lemma shows that $f$ is a well-defined mapping.

\begin{lemma}\label{map}For each Clar cover $C\in \mathbb{Z}(H,n)$, we have $f(C)\in \mathbb{Q}_n(R(H))$.
\end{lemma}
\begin{proof}It is sufficient to show that $f(C)$ is isomorphic to the $n$-cube $Q_n$. Let $h_1,h_2,...,h_n$ be the hexagons of $C$. For any vertex $M$ of $f(C)$ ($M$ is also a perfect matching of $H$), let $b(M)=b_1b_2\cdots b_n$, where $b_i=1$ or 0 according to whether $h_i$ is a proper or improper $M$-alternating hexagon with $i = 1,2, \dots n$. It is obvious that $b:V(f(C))\rightarrow V(Q_n)$ is a bijection. For $M'\in V(f(C))$, let $b(M')=b_1'b_2'\cdots b_n'$. If $M$ and $M'$ are adjacent in $f(C)$, which means they adjacent in $R(H)$, then $M\oplus M' = h_i$ for some $i \in [1,n]$. Therefore, $b_j=b_j'$ for each $j\not=i$ and $b_i\not=b_i'$, which implies $b_1b_2\cdots b_n$  and $b_1'b_2'\cdots b_n'$ are adjacent in $Q_n$. Conversely, if $b_1b_2\cdots b_n$  and $b_1'b_2'\cdots b_n'$ are adjacent in $Q_n$, it follows that $M$ and $M'$ are adjacent in $f(C)$. Hence $b$ is an isomorphism between $f(C)$ and $Q_n$. \end{proof}

\begin{figure}
\centering
\includegraphics[scale=0.6]{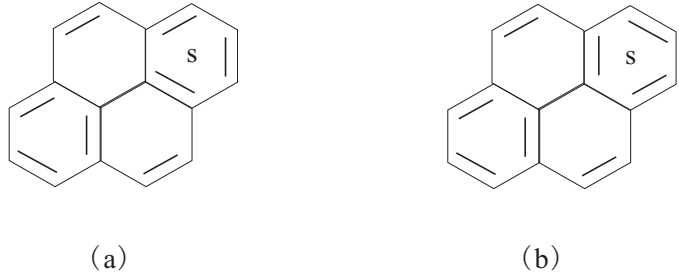}
\caption{(a) A proper sextet $s$, (b) An improper sextet $s$.}\label{proper}
\end{figure}
\begin{lemma} \label{inj} The mapping $f: \mathbb{Z}(H,n)\rightarrow\mathbb{Q}_n(R(H))$ is injective for each nonnegative integer $n$.
\end{lemma}
\begin{proof}  Given any distinct Clar covers $C$ and $C'$ in $\mathbb{Z}(H,n)$. If $C$ and $C'$ contain the same set of hexagons, then the isolated edges of $C$ and $C'$ are distinct. Therefore, $f(C)$ and $f(C')$ are disjoint induced subgraphs of $R(H)$ and thus  $f(C)\not=f(C')$. Suppose $C$ and $C'$ contain different sets of hexagons and let $h$ be a hexagon in $C \setminus C'$. Hence there at least one edge $e$ of $h$ does not belong to $C'$. From the definition of function $f$, $e$ is thus unsaturated by those perfect matchings that corresponding to the vertices in $f(C')$. However, there exist vertices $M_1$ and $M_2$ of $f(C)$ ($M_1$ and $M_2$ are perfect matchings of $H$) such that $M_1\oplus M_2=E(h)$ because $h$ is a hexagon in $C$. Hence $e$ is saturated by one of $M_1$ and $M_2$, say $M_1$. As a result, $M_1\notin V(f(C'))$ and $f(C)\not=f(C')$.
\end{proof}

The following  lemma is an obvious and known result.
\begin{lemma}\label{proper} For a perfect matching $M$ of $H$, the proper (resp. improper) $M$-alternating hexagons are pairwise disjoint.
\end{lemma}

\begin{lemma}[\cite{SKG}]\label{4-cycle} If the resonance graph $R(H)$ contains a 4-cycle $M_1M_2M_3M_4M_1$, then $h:=M_1\oplus M_2$ and $h':=M_1\oplus M_4$ are disjoint  hexagons. Also, we have $h=M_3\oplus M_4$ and $h'=M_2\oplus M_3$.
\end{lemma}
\begin{proof}Since $M_1M_2M_3M_4M_1$ is  a 4-cycle, the $h_i:=M_i\oplus M_{i+1}$, where the subscripts are module 4, are hexagons of $H$ and $h_i\not= h_{i+1}$ for each $i$. Moreover, the set $h_0\oplus h_1\oplus h_2\oplus h_3=\emptyset$ because symmetric difference is commutative and associative.  It is known that any two distinct hexagons in $H$ are either disjoint or having exactly one edge in common. If $h_0\not=h_2$, then $h_0$ has at most three edges in $E(h_1)\cup E(h_2)\cup E(h_3)$ and all other edges are contained in the above symmetry difference, which leads to a contradiction. Hence $h_2=h_0$ and using similar arguments, we have $h_1=h_3$ as well as $h_0$ and $h_1$ are disjoint.
\end{proof}

We now define an orientation of the directed resonance graph $\vec R(H)$: an edge $MM'$ is oriented from  $M$ to $M'$ if  $M\oplus M'$ is a proper sextet with respect to $M$ and an improper sextet with respect to $M'$. In fact $\vec R(H)$ is the Hasse diagram of a distributive lattice on the set of perfect matchings of $H$ \cite{z3}.

\begin{lemma}[\cite{z3,ZLS}]\label{directed} The directed resonance graph $\vec R(H)$ has no directed cycle.
\end{lemma}

\begin{lemma}\label{sur} The mapping $f: \mathbb{Z}(H,n)\rightarrow\mathbb{Q}_n(R(H))$ is  surjective for each nonnegative integer $n$.
\end{lemma}
\begin{proof} For any graph $G_n\in \mathbb{Q}_n(R(H))$, which is isomorphic to the $n$-cube $Q_n$, the corresponding oriented graph $\vec G_n$ in $\vec R(H)$ has no directed cycle by Lemma \ref{directed}.  Hence $\vec G_n$ contains a vertex $M_0$ (a source) with in-degree 0 and  out-degree $n$,  that is,  $\vec G_n$ has $n$ directed edges from $M_0$ to $M_1, M_2,\ldots,$ and $M_n$. Let $h_i=M_0\oplus M_i$ for $i=1,2,...,n$, the directed edges $M_0M_i$ implies that all $h_i$ are proper $M_0$-alternating hexagons. Moreover, Lemma \ref{proper} guarantees that all $h_i$ are pairwise disjoint. Therefore, we can obtain a Clar cover $C$ in $\mathbb{Z}(H,n)$ by regarding the $M_0$-alternating hexagons $h_1,h_2,...,h_n$ as components and all others edges of $M_0$ as isolated edge components.  It suffices to prove $f(C)=G_n$.

Let $[n]=\{1,2,...,n\}$ and for any $I\subseteq [n]$, let $\displaystyle M_I:=M_0\oplus (\displaystyle\bigcup_{i\in I}h_i)$. Then $V(f(C))=\{M_I: I\subseteq [n]\}$.  In particular, $M_0=M_{\emptyset}$ as well as $M_i=M_{\{i\}}$ for all $i$, and all of them belong to both $G_n$ and $f(C)$. On the other hand, since $G_n$ is isomorphic to the $n$-cube $Q_n$, every vertex $F$ of $G_n$ can be labeled with binary strings $b(F)=b_1b_2\cdots b_n$ such that $b(M_0)=00\cdots 0$ and for each $b(M_i)$, the $i$-th coordinate is 1 and the other coordinates are all zeroes. Hence, $F$ and $F'$ are adjacent in $G_n$ if and only if $b(F)$ and $b(F')$ are adjacent in $Q_n$.

For any $F\in V(G_n)$, let $I_F =\{i\in[n]: b_i=1\}$. We will show that $F=M_{I_F}\in V(f(C))$ by induction  on the distance $|I_F|$ between $M_0$ and $F$ in $G_n$. If $|I_F|=0$ or 1, then $F\in \{M_0, M_1, \dots M_n\}$. Now suppose $|I_F|\geq 2$, there exists $i,j \in I_F$ with $i\not=j$. Thus, there are vertices $F_0, F_1, F_2$ of $G_n$ such that $I_{F_0}=I_F\setminus \{i,j\}$, $I_{F_1}=I_F\setminus \{i\}$, and $I_{F_2}=I_F\setminus \{j\}$. By the induction hypothesis, we have $F_k=M_{I_{F_k}}$ for $k=0,1$ and 2. Therefore, $F_1\oplus F_0=M_{I_{F_1}}\oplus M_{I_{F_0}}=h_j$ and $F_2\oplus F_0=M_{I_{F_2}}\oplus M_{I_{F_0}}=h_i$. Since $F_0F_1FF_2F_0$ is a 4-cycle of $G_n$, it follows from Lemma \ref{4-cycle} that $F\oplus F_1=F_2\oplus F_0=h_i$. As a result, $ F=F_1\oplus h_i=(M_0\oplus ({\displaystyle\bigcup_{ k\in I_{F_1}} } h_k))\oplus h_i=M_0\oplus (\displaystyle\bigcup_{k\in I_{F}}h_i)=M_{I_F}$ and this implies that $V(f(C))=V(G_n)$ and hence $f(C)=G_n$.
\end{proof}

Combining Lemmas \ref{map}, \ref{inj} and \ref{sur}, we have the following corollary immediately.

\begin{cor}\label{bij} The function $f: \mathbb{Z}(H,n)\rightarrow\mathbb{Q}_n(R(H))$ is a bijection for every nonnegative integer $n$.
\end{cor}

Hence, we obtain the main result of this article as follows:

\begin{thm}\label{identity}For any Kekul\'ean hexagonal system $H$, we have $\zeta(H,x)=C(R(H),x)$.
\end{thm}

 Corollary \ref{bij} also implies the following result.

\begin{cor}[\cite{SKG}] Let $H$ be a Kekul\'ean hexagonal system. For any non-negative integer $k$, there exists a surjective mapping from the set of $k$-cubes of $R(H)$ to the set of resonant sets with $k$ hexagons.
\end{cor}

\section{Maximal hypercubes}

We now consider maximal hypercubes in resonance graphs by establishing ordering relations on the hypercubes and the Clar covers.

For a hexagonal system $H$ with perfect matchings, let $\mathbb{Q}$ be the set of induced subgraphs of $R(H)$ that are hypercubes, thus $\mathbb{Q}= \displaystyle\bigcup_{n\geq 0}\mathbb{Q}_n(R(H))$. We define an ordering $\leq$ on $\mathbb{Q}$ as follows:  For any graphs $Q,Q'\in \mathbb{Q}$, we have
 $Q\leq Q'$ if $Q$ is a subgraph of $Q'$. Hence $(\mathbb{Q}, \leq)$ is a poset and the maximal hypercubes of $R(H)$ are the maximal elements of the poset $(\mathbb{Q}, \leq)$.

 Let $\mathbb{C}$ denote the set of all Clar covers of $H$, thus $\mathbb{C}=\displaystyle\bigcup_{n\geq 0}\mathbb{Z}(H,n)$. For any Clar covers $C,C'\in \mathbb{C}$, we have
 $C\leq C'$ if $f(C)\subseteq f(C')$. The reflexivity and transitivity of this binary relation are obvious and its antisymmetry follows from the bijection $f$ (Corollary \ref{bij}). As a result, $(\mathbb{C}, \leq)$ is also a poset.

 Moreover, the bijection $f$ from $(\mathbb{C}, \leq)$ to $(\mathbb{Q}, \leq)$  preserves the ordering relations and we obtain the following result.

 \begin{thm}The posets $(\mathbb{C}, \leq)$ and $(\mathbb{Q}, \leq)$ are isomorphic.
 \end{thm}

 We now give an explicit description for the partial ordering on the Clar covers of $H$.

 \begin{thm}\label{order}For any Clar covers $C,C'\in \mathbb{C}$, $C\leq C'$ if and only if the hexagons of $C$ belong to $C'$  as well as $C$ and $C'$ coincide on the edges apart from those in the hexagons of $C'$.
 \end{thm}

 \begin{proof}If the hexagons of $C$ belong to $C'$  as well as $C$ and $C'$ coincide on the edges apart from those in the hexagons of $C'$, then the hexagons in $C'\setminus C$ are alternating in $C$. Hence $V(f(C))\subseteq V(f(C'))$ and $f(C)\subseteq f(C')$, that is $C\leq C'$. Conversely, suppose   $f(C)\subseteq f(C')$. For any hexagon $h$ in $C$, there are two perfect matchings $M$ and $M'$ of $H$ in $f(C)$ such that $h=M\oplus M'$. Since $M$ and $M'$ are in $f(C')$,  $h$ is also in $C'$. Moreover, it is obvious that $C$ and $C'$ coincide on the edges apart from those in the hexagons of $C'$.
 \end{proof}

 \begin{figure}
\centering
\includegraphics[scale=0.6]{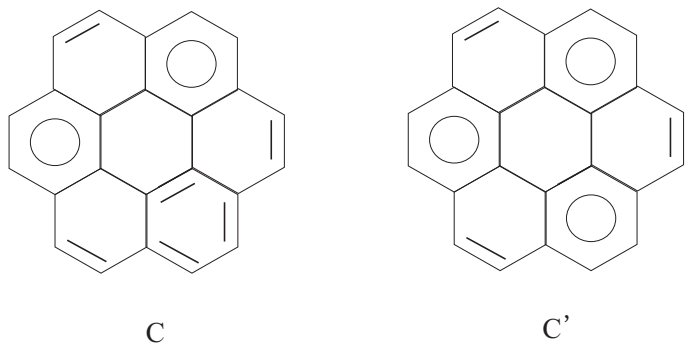}
\caption{Two Clar covers of coronene: $C\leq C'$.}\label{fig_order}
\end{figure}

The theorem shows that a Clar cover $C$ of $H$ is maximal if and only if $C$ has no alternating hexagons in $H$. Accordingly, the isomorphism $f$ between $(\mathbb{C}, \leq)$ and $(\mathbb{Q}, \leq)$ implies the following corollary.

 \begin{cor}There is a one-to-one correspondence between the maximal hypercubes of $R(H)$ and the Clar covers of $H$ that without alternating hexagons.
 \end{cor}

Relations between  Clar structures and Clar covers without alternating hexagons have already been discussed in \cite{zzg}.

\section{Some applications}

The Clar covering polynomial and cube polynomial were studied independently in the past. The Clar covering polynomials of many types of hexagonal systems have been obtained explicitly via various approaches \cite{CDG,CLW,CW,CW2,GB,gfb,ZJY,zz1}. Hence, the cube polynomial of their resonance graphs can be obtained by Theorem \ref{identity}.
For example,  $$\zeta(\textrm{pyrene},x)=C(R(\textrm{pyrene}),x)=x^2+6x+6=(x+1)^2+4(x+1)+1,$$ and
$$\begin{array}{cl}
\zeta(\textrm{coronene},x)=C(R(\textrm{coronene}),x) & =2x^3+15x^2+32x+20 \\
 &=2(x+1)^3+9(x+1)^2+8(x+1)+1, \end{array}$$
where the resonance graph $R(\textrm{coronene})$ is illustrated in \cite{ZLS}.

In the following subsections we will present some interesting applications of Theorem \ref{identity}.

Although the cube polynomial was defined for any graphs, its discussions have been concentrated on median graphs that hypercubes play an important role.  A median of a triple of vertices $u, v$ and $w$ of a graph is a vertex that lies on a shortest  $u,v$-path, $u,w$-path  and  $v,w$-path simultaneously. A graph is called a {\em median graph} if every triple of its vertices has a unique median. Zhang et al. \cite{ZLS} proved that the resonance graph of a (weakly) elementary plane bipartite graph is a median graph, which is done by considering a distributive lattice structure on the set of its perfect matchings. A median graph, however, is not necessarily a resonance graph.

\subsection{Fibonacci cube}
A fibonacene is a hexagonal chain in which no hexagons are linearly attached. Klav\v{z}ar and \v{Z}igert \cite{KZ} showed that the resonance graph of a fibonacene with $n$ hexagons $Z_{n}$ is isomorphic to Fibonacci cube $\Gamma_n$, which is the subgraph of the $n$-cube $Q_n$ induced by all binary strings of length $n$ that contain no two consecutive 1s. The Fibonacci cubes were introduced \cite{H} as a model for interconnection networks and  Klav\v{z}ar \cite{K1} gave an extensive researches on it. Also, the chemical graph theory of fibonacenes was studied in \cite{GK}. The Clar covering polynomial of a fibonacene $Z_{n}$ was expressed \cite{z2} in terms of binomial coefficients by matching polynomial of a path with $n+1$ vertices; the cube polynomial of $\Gamma_n$ was derived from generating functions with double variables \cite{KM}. Indeed, they have the same expression:

\begin{equation}\zeta(Z_{n}, x)=C(\Gamma_n,x)=\sum_{k=0}^{\lfloor\frac{n+1}{2}\rfloor}\Big(\begin{array}{c}
                                                       n-k \\
                                                        k
                                                      \end{array}\Big)
(x+1)^k.
\end{equation}

\noindent Besides, the sextet polynomial of $Z_n$ is

$$B(Z_{n}, x)=\sum_{k=0}^{\lfloor\frac{n+1}{2}\rfloor}\Big(\begin{array}{c}
                                                       n-k \\
                                                        k
                                                      \end{array}\Big)
x^k.$$

\subsection{Derivatives of Clar covering polynomials}

 Bre\v{s}ar et al.\cite{BKS} studied the derivatives of  cube polynomials of median graphs, which can be expressed as the cube polynomial of the disjoint union of some subgraphs with median property. Motivated by this, we consider the derivative of the Clar covering polynomial of hexagonal systems directly. 

\begin{thm}If $H$ is a generalized hexagonal system, then $\displaystyle \zeta'(H,x)=\sum_{h}\zeta(H-h,x),$
where the summation goes over all hexagons $h$ of $H$.\end{thm}
\begin{proof} We count the ordered pairs $(C,h)$, where $C$ is a Clar cover of $H$ with $k$ hexagons and $h$ is a hexagon of $C$, in two different ways. Since each Clar cover $C$ is counted $k$ times, we obtain the number $kz(H,k)$ of ordered pairs. On the other hand, when $h$ is fixed, the number of covers $C$ containing $h$ is equals to $z(H-h,k-1)$. Hence, the total amount is $\displaystyle\sum_h z(H-h,k-1)$. Therefore,  $\displaystyle kz(H,k)=\sum_h z(H-h,k-1)$ for each positive integer $k$ and the result follows.
\end{proof}

The theorem can be applied repeatedly and we obtain the following corollary about high  derivatives of  Clar covering polynomials of hexagonal systems.

\begin{cor} If $H$ is a hexagonal system, then $\zeta^{(s)}(H,x)=\sum_{\mathcal{R}_s}\zeta(H-\mathcal{R}_s,x),$ where the summation goes over all sextet patterns $\mathcal{R}_s$ with $s$ hexagons.
\end{cor}

\subsection{Roots of Clar covering polynomials}
Gutman et al. \cite{ggf} found good linear correlations
between topological resonance energy  and $\ln \zeta(H,x)$ for fixed values of $x$ lying in the
interval [0, 2].  Gojak et al. \cite{GGRV} further implemented a model relating resonance energy  with  $\sqrt{\zeta(H,x)}$ for  suitable values of $x$. Hence the real roots of Clar covering polynomial is helpful in such researches.

Bre\v{s}ar et al. \cite{BKS2} ever obtained some important properties of   real roots of cube polynomials of median graphs as follows.
\begin{thm}\cite{BKS2}\label{root} Let $G$ be a median graph. Then $C(G,x)$ has no roots in $[-1,+\infty)$.
\end{thm}

\begin{thm}\cite{BKS2} Let $G$ be a median graph. If $C(G,a)=0$ for a rational number $a$, then $a=-\frac{t+1}{t}$ for some integer $t\geq 1$.
\end{thm}

\begin{thm}\cite{BKS2}\label{root2} Let $G$ be a nontrivial median graph. Then  $C(G,x)$ has a real root in the interval $[-2,-1)$.
\end{thm}

By Theorem \ref{identity} together with Theorems \ref{root} to \ref{root2} we immediately have the following results  for the real roots of Clar covering polynomials as follows.

\begin{cor}Let $H$ be a Kekulean hexagonal system. If  a rational number $a$ is a root of $\zeta(H,x)$, then $a=-\frac{t+1}{t}$ for some integer $t\geq 1$.
\end{cor}

\begin{cor}Let $H$ be a Kekulean hexagonal system. Then  $\zeta(H,x)$ has no roots in $[-1,+\infty)$ and has a real root in the interval $[-2,-1)$.
\end{cor}

\subsection{A transformation of cube polynomials}

We have known that Clar covering polynomial of a hexagonal system can be expressed in the power of $(x+1)$ such that the coefficients are nonnegative as follows.

\begin{thm}\label{cha02 4}\cite{zz2}
Let $H$ be a Kekul\'ean hexagonal system. Then
$$\zeta(H,x)=\sum_{i=0}^{ Cl(H)}z(H,i)x^{i}=\sum_{i=0}^{Cl(H)}
a(H,i)(x+1)^{i},$$
where $a(H,i)$ denotes the number of perfect matchings of $H$ with exactly $i$ proper sextets.
\end{thm}

This result enable one to know that the cube polynomial of resonance graph have the same result.  So a question naturally arises: does the cube polynomial of a general median graph have such an expression?  We shall give an affirmative answer to this question in this subsection.

To this end we introduce a convex expansion of median graphs. A pair of induced subgraphs $\{G_1,G_2\}$ of a graph $G$ is called a {\em cubical cover} if $G=G_1\cup G_2$ and each induced hypercube lies in at least one of $G_1$ and $G_2$. The {\em expansion} $G^*$ of $G$ with respect to the cubical cover $\{G_1,G_2\}$ of $G$ is the graph obtained from the disjoint union of copies $G_1^*$ and $G_2^*$ of $G_1$ and $G_2$ by adding an edge between both corresponding vertices in $G_1^*$ and $G_2^*$ of each vertex of $G_0=G_1\cap G_2$.
\begin{pro}\cite{BKS}\label{recursive} Let $G^*$ be a graph constructed by the expansion with respect to the cubical cover $\{G_1,G_2\}$ (over $G_0$). Then $C(G^*,x)=C(G_1,x)+C(G_2,x)+xC(G_0,x)$.
\end{pro}

An expansion with respect to $\{G_1,G_2\}$ is called a {\em peripheral convex expansion} if $G_1$ is a convex subgraph of $G_2$ (i.e. for any pair of vertices $u$ and $v$ in $G_1$, any shortest $u,v$-path of $G_2$ lies completely in $G_1$). The following result gives a construction for median graphs via such a convex expansion.

\begin{thm}\cite{BKS2}\label{median} Let $G$ be a connected graph. Then $G$ is a median graph if and only
if $G$ can be obtained from the one vertex graph by a sequence of peripheral convex
expansions.
\end{thm}

\begin{thm}\label{expression}Let $G$ be a median graph with the maximum dimension $m$ of hypercubes contained in $G$. Then $C(G,x)=\sum_{i=0}^mb_i(G)(x+1)^i$, where $b_0(G)=1$ and $b_i(G)$ is a positive integer for each $0\leq i\leq m$.
\end{thm}

\begin{proof}   We proceed by induction on the number of vertices of median graph $G$. Obviously, $C(K_1,x)=1$ and $C(K_2,x)=2+x=1+(x+1)$, which show the basis of induction. So suppose  $G$ is a median graph with more than $K_2$. By Theorem \ref{median} $G$ can be constructed from a median graph $G'$ by a peripheral convex expansion with respect to $G_0$. By Proposition \ref{recursive}, we have
$$C(G,x)=C(G',x)+(x+1)C(G_0,x).$$
Since $G'$ and $G_0$ are median graphs which are smaller than $G$, by induction hypothesis we have $C(G',x)=\sum_{i\geq 0}b_i(G')(x+1)^i$ and $C(G_0,x)=\sum_{i\geq 0}b_i(G_0)(x+1)^i$ satisfying the conditions of the theorem.  Then $C(G,x)=\sum_{i\geq 1}(b_i(G')+b_{i-1}(G_0)) (x+1)^i+b_0(G')$. We can see $b_0(G)=b_0(G')=1$, and $b_i(G)=b_i(G')+b_{i-1}(G_0)$ is a positive integer for each $i\leq m$.
\end{proof}

Further we will determine the coefficients $b_i(G)$ of this novel expression of cube polynomial $C(G,x)$ and give their combinatorial explanation.

For a median graph $G$,  from Theorem \ref{expression} we have the following inversion formulas.
\begin{cor}\label{inv}
(i) $\alpha_i(G)=\sum_{k=0}^mb_k(G)\Big(\begin{array}{c}k \\i\end{array}\Big),i=0,1,2,...,m,$ and\\
(ii) $b_j(G)=\sum_{k=0}^m(-1)^{k-j}\Big(\begin{array}{c}k \\j\end{array}\Big)\alpha_k(G),j=0,1,2,...,m.$
\end{cor}

On the other hand, Bre\v{s}ar et al. \cite{BKS} introduced high derivative graph $\partial^k G$ of a median graph $G$ and obtained that $C^{(k)}(G,x)=C(\partial^k G,x)$.  Let  $\theta_i(G)$ denote the number of components in $\partial^i(G)$,  $i\geq 0$.   Corollary 10 of \cite{BKS} shows that for each $i\geq 0$, $\theta_i(G)$ can be expressed as
$$\theta_i(G)=i! \sum_{k\geq 0}(-1)^{k-i}\Big(\begin{array}{c}k \\i\end{array}\Big)\alpha_k(G).$$
Compared the above equation with the Eqs.(ii)  in Corollary \ref{inv}, we immediately have $b_i(G)=\theta_i/i!,$ and have the following results.
\begin{cor}\label{expression'}Let $G$ be a median graph. Then $C(G,x)=\sum_{i\geq 0}\frac{\theta_i}{i!}(x+1)^i$.
\end{cor}

\begin{cor}For a Kekul\'ean hexagonal system $H$, we have that  for $i\geq 0$
\begin{equation}i!a(H,i)=\theta_i(R(H)).\end{equation}
\end{cor}

As an application of Theorem \ref{expression} we now discuss the coefficients  $\alpha_i(G)$ of cube polynomial. First Corollary \ref{inv}(ii) together with $b_0(G)=1$ immediately implies the ollowing well-known result.

\begin{cor}\cite{BKS} Let $G$ be a median graph. Then $\sum_{i\geq 0}(-1)^i\alpha_i(G)=1$.
\end{cor}

In an analogous manner as \cite{zz2}, applying Corollary  \ref{inv}(i) we can obtain a monotonic section of the sequence of  coefficients of cube polynomials as follows.

\begin{cor}Let $G$ be a median graph. Then $\alpha_m(G)<\alpha_{m-1}(G)<\cdots<\alpha_{\lceil\frac{m-1}{2}\rceil}(G)$.
\end{cor}

We now propose  the following conjecture for a median graph. When restricted on the Clar covering polynomial of hexagonal systems, the same conjecture was put forward in \cite{zz2} and remains unsolved.

\begin{conj}The sequence of coefficients of cube polynomial of a median graph
are unimodal.
\end{conj}

\end{document}